\documentclass[11pt]{amsart}
\usepackage{amsmath,amssymb,latexsym,soul,cite,mathrsfs,amsfonts}
\usepackage{subfigure}
\usepackage{float}
\usepackage{caption}
\usepackage{xcolor}
\usepackage{color,enumitem,graphicx}
\usepackage[colorlinks=true,urlcolor=blue,
citecolor=red,linkcolor=blue,linktocpage,pdfpagelabels,
bookmarksnumbered,bookmarksopen]{hyperref}
\usepackage[english]{babel}
\usepackage{outline}
\usepackage{comment}
\usepackage[left=2.9cm,right=2.9cm,top=2.8cm,bottom=2.8cm]{geometry}
\usepackage[hyperpageref]{backref}
\usepackage[colorinlistoftodos]{todonotes}
\makeatletter
\providecommand\@dotsep{5}
\def\listtodoname{List of Todos}
\def\listoftodos{\@starttoc{tdo}\listtodoname}
\makeatother

\usepackage{verbatim} 

\numberwithin{equation}{section}

\newtheorem{theorem}{Theorem}[section]
\newtheorem{prop}[theorem]{Proposition}
\newtheorem{lem}[theorem]{Lemma}

\newtheorem{rem}{Remark}
\newtheorem{definition}{Definition}
\newcommand\restr[2]{{
		\left.\kern-\nulldelimiterspace 
		#1 
		\vphantom{\big|} 
		\right|_{#2} 
}}

\title[Mixed local and non-local critical problem]
{Multiplicity result 
for mixed local and nonlocal Kirchhoff  problem  involving critical growth
}

\author[Vinayak Mani Tripathi]{V. M. Tripathi}

\address[Vinayak Mani Tripathi]{\newline\indent
		Department of Mathematics
		\newline\indent 
		Indian Institute of Technology Bhilai
		\newline\indent
		491002, Durg, Chhattisgarh, India}
\email{\href{mailto:vinayakm@iitbhilai.ac.in}{vinayakm@iitbhilai.ac.in}}


\begin{document}

	\begin{abstract}
		In this paper, we study the  multiplicity of nonnegative solutions for mixed local and non-local problem involving critical nonlinearity with sign changing weight. Using Nehari manifold method and fibering map analysis, we have shown existence of  two solutions.
	\end{abstract}
	
	\maketitle

	\section{Introduction}
	In this work we study the multiplicity of nonnegative solutions for the following mixed local and non-local Kirchhoff problem
	\begin{equation}\label{p}	\tag{$P_{\lambda}$}
		\left\{
		\begin{aligned}
			M\left(\int_{\mathbb{R}^N}|\nabla u|^2dx+\int\int_{\mathbb{R}^{2N}}\frac{|u(x)-u(y)|^2}{|x-y|^{N+2s}}dxdy\right)\mathcal{L}(u)&= \lambda {f(x)}|u|^{p-2}u+|u|^{2^*-2}u \hspace{0.9cm} \mbox{in}\ \ \Omega, \\
			u&=0                                  \hspace{3.8cm}\mbox{on}\ \ \mathbb R^N\setminus \Omega,
		\end{aligned}
		\right.
	\end{equation}where, $\Omega\subset\mathbb{R}^N$ is bounded domain with smooth boundary, $1<p<2<2^*=\frac{2N}{N-2}, N\geq 3,\lambda>0$ is a positive parameter and $\mathcal{L}$ denotes the mixed local and nonlocal operator
	\[\mathcal{L}(u):=-\Delta u+ (-\Delta)^su,  \ s\in (0,1).\] The weight  function $f\in L^\frac{2^*}{2^*-p}(\Omega)$  is allowed to change  sign.
	The fractional Laplacian operator $(-\Delta)^s $ is defined, up to a normalization constant, as
	\[(-\Delta)^s\phi(x)=\int_{\mathbb{R}^{N}}\frac{2\phi(x)-\phi(x+y)-\phi(x-y)}{|y|^{N+2s}}dy, \ \ x\in \mathbb{R}^N\] for any $\phi\in C^\infty_0(\Omega)$ (see, \cite{hguide} for more details). The map $M:[0,\infty)\to [0,\infty)$ is continuous function defined by
\begin{equation}\label{mt}
	M(t)=a+bt^{\theta-1},
\end{equation} where $\theta\in[1,\frac{2^*}{2})$, $a> 0$ and $b>0$.
	
	The above class of problems can be seen as stationary state of the following problem
	\begin{equation*}
		\left\{
		\begin{aligned}
			u_{tt}-M\left(\int_{ \Omega}|\nabla u|^2dx\right)\Delta u&= F(x,u) &&\mbox{in}\ \ \Omega, \\
			u&=0                                   &&\mbox{on}\ \ \partial\Omega,
		\end{aligned}
		\right.
	\end{equation*}
	which was initially introduced by Kirchhoff (see \cite{kirchhoff}) to address free transversal oscillations of elastic strings. The term $M$ measure the stress in the string resulting from change in its length during vibration. It is directly proportional to the Sobolev norm of the string displacement. Nonlocal diffusion problems have received a lot of attention in recent past, especially those that are driven by the fractional Laplace operator. A potential reason for  this is that this operator naturally arises in a number of physical phenomena including population dynamics, geophysical fluid dynamics, flame propagation and liquid chemical reactions. Furthermore, in probability theory, it provides a basic model to explain specific jump Lévy processes (see, \cite{prob1, valdi,twod}). The study of a fractional Kirchhoff model arising from the analysis of string vibrations has been discussed in \cite{fiscella}. 
	
	 A lot of attention for local and non-local elliptic problems  has been given to study   existence and multiplicity of solutions dealing with nonlinearity, as involved in the problem \eqref{p}. After the pioneering work of Ambrosetti et al. in\cite{abc}, many researchers have shown their interest to the particular class of nonlinearity that is the subject of this study, known as concave-convex nonlinearity. For $p-$Laplacian and fractional operator  we refer interested readers to see the work in \cite{Peral,azorero,barios} and the references therein,  where the authors have studied such class of problems. There are numerous papers dealing with Kirchhoff type nonlocal problems involving fractional Laplacian.   With no attempt to provide the complete list of references, we cite \cite{bar1,pucci,b2,caff,fiscella,pucci1,servd,servd2} and  references therein for problems involving fractional  Laplacian and fractional Kirchhoff operator. For recent developments in nonlocal Kirchhoff problems, we suggest interested readers to see the survey \cite{puccird}. In \cite{tfwu,mana} authors have consider semilinear and fractional Kirchhoff problem with concave-convex critical nonlinearity   involving sign changing weight. Using Nehari manifold idea (see,\cite{n1,n2}) authors proved the existence of at least two positive solutions for suitable choice of  parameter $\lambda$.
	
	The mixed operator have been the subject of the study in the recent past. This naturally arises from the superposition of two different-scale stochastic processes: a Lévy flight and a classical random walk.  For instance, in relation to the study of the best foraging techniques and the spread of biological species \cite{valdi1}, the species survival problem \cite{valdinoci}, and  so on. Inspired from interesting application of mixed local and local and nonlocal problems the following class of problem is addressed in the literature for analysis of several qualitative properties of solutions
	\[-\Delta_qu+(-\Delta)^s_qu=f, \ \ \textrm{in} \ \Omega\]    where  $s\in(0,1),q\in (1,\infty), -\Delta_q$ and $(-\Delta)^s_q$ denotes $q-$Laplacian and fractional $q-$Laplacian respectively. In case  $q=2$, we refer interested readers to see the works in \cite{cozzi,barles,biagi2} wherein the authors have shown the existence of weak solutions, the strong maximum principle, local boundedness, interior Sobolev, and Lipschitz regularity.   By using variational techniques, authors in \cite{mugnai} exhibited the existence of a weak solution when  $f$ has  at most linear growth. Boundedness and strong maximum principle for the inhomogenous case have been proved in \cite{biagi3}. In the case $q\in (1,\infty)$ and $f=0$,  the local boundedness of weak subsolutions, local Hölder continuity of weak solutions, Harnack inequality for weak solutions and weak Harnack inequality for weak supersolutions has been established in \cite{garain1}. Additionally, we suggest interested readers to view the results in \cite{biagi4,biagi5,garain2,d2,fiscmixed}  and the   references therein for problems related with  mixed local and nonlocal	operators.
	
{	In this paper we aim to show multiplicity result for problem \eqref{p}. The major challenge here is to show existence of second solution for \eqref{p} as the optimal constant in mixed type Sobolev inequality is not achieved (see, \cite{biagi1}). By using Talenti functions, Brezis-Lieb result, and estimates from \cite{biagi1}, we have ensured the existence of second solution under some restriction on dimension and on the fractional parameter. To the best of our knowledge this is the first attempt to study concave-convex class of problem for mixed local and nonlocal operators using Nehari manifold idea. }
	
	This work is organized into the following sections. In section \ref{prem}, we recall all the relevant notation, definitions and preliminary results used throughout this work. In section \ref{vf} we have provided the framework of the Nehari  setup and some technical results. We have shown the existence of minimizers for the energy functional in sections \ref{sec4} and \ref{sec5}.  
	We conclude this section by stating the main results of our work.
	\begin{theorem}\label{t1}
		There exists $\Lambda_0>0$ such that the problem \eqref{p} has a nontrival nonnegative  solution  for all $\lambda\in(0,\Lambda_0)$.	\end{theorem}
		\begin{theorem}\label{t2} Let $N+4s<6$, then there 
			 exists $0<\Lambda_{00}\leq \Lambda_0$, such that  the problem \eqref{p} has at least two  nontrival nonnegative  solutions  for all $\lambda\in (0,\Lambda_{00})$ and for sufficiently small values of $b$ in \eqref{mt}. 
		\end{theorem} 
		\begin{rem}As $N+4s<6$ in Theorem  \ref{t2} we can emphasize that the results hold good, only in dimension 3, 4 and 5 under certain restriction on the fractional exponent.
		\end{rem}
			\section{Preliminaries}\label{prem} Let   $s\in(0,1)$. If $u:\mathbb{R}^N\to \mathbb{R}$ is a measurable function, we define 
	\[[u]_s:=\left(\int\int_{\mathbb{R}^{2N}}\frac{|u(x)-u(y)|^2}{|x-y|^{N+2s}}dxdy\right)^\frac{1}{2},\] the so called Gagliardo semi norm of $u$ of order $s$. Let $\Omega\subset\mathbb{R}^N$ is an open set (with Lipschitz boundary), not necessarily bounded. We denote $\mathcal{X}^{1,2}(\Omega)$ as the completion of $C^\infty_0(\Omega)$ with respect to the following global norm 
	\[\rho(u):=\left(\|\nabla u\|^2_{L^2(\mathbb{R}^N)}+[u]^2_s\right)^\frac{1}{2}, \ \ \ u\in C^\infty_0(\Omega).\] 
	\begin{rem}
		A few remarks are in order concerning the space $\mathcal{X}^{1,2}(\Omega)$.
		\begin{enumerate}
			\item The norm $\rho(u)$ is induced by the scalar product 
			\[<u,v>_\rho:=\left(\int_{\mathbb{R}^N}\nabla u\cdot\nabla vdx+\int\int_{\mathbb{R}^{2N}}\frac{(u(x)-u(y))(v(x)-v(y))}{|x-y|^{N+2s}} dxdy\right),\] where $\cdot$ denotes the usual scalar product on the Euclidean space $\mathbb{R}^N$ and $\mathcal{X}^{1,2}(\Omega)$ is a Hilbert space. 
			\item Despite of $u\in C^\infty_0(\Omega)$ the $L^2-$norm of $\nabla u$  is considered  on the  whole space. This is to emphasize that elements in $\mathcal{X}^{1,2}(\Omega)$ are functions defined on the entire space $\mathbb{R}^N$ and not only in $\Omega$. The benefit of having this global functional setting is that these functions can be globally approximated on $\mathbb{R}^N$ (with respect to the norm $\rho(\cdot)$) by smooth functions with compact support in $\Omega$. 
			
			In particular when $\Omega\neq \mathbb{R}^N$, we can see that this global definition of $\rho(\cdot)$, implies that the functions in $\mathcal{X}^{1,2}(\Omega)$ naturally satisfy the nonlocal Dirichlet condition specified in problem \eqref{p}, that is, 
			\begin{equation}\label{space}
				u\equiv0 \ \ \textrm{a.e.} \ \textrm{in} \ \mathbb{R}^N\setminus\Omega \ \textrm{for} \ \textrm{every} \ u\in \mathcal{X}^{1,2}(\Omega).
			\end{equation}{For detail understanding  of nature of the space $\mathcal{X}^{1,2}(\Omega)$ and the validity of  \eqref{space}, we refer interested readers to  see the remark 2.1 in \cite{biagi1,biagi5}.}
			\item The embedding  $\mathcal{X}^{1,2}(\Omega)\hookrightarrow L^r(\Omega)$ is compact  for every $r\in [1,2^*)$.
		\end{enumerate} 
	\end{rem}
	The meaning of solution of the underline problem is stated as follows. 
	
		\begin{definition}\label{def}
		A function $u\in \mathcal{X}^{1,2}(\Omega)$ is said to be a (weak) solution of the problem \eqref{p} if 
		\[M(\rho( u)^2)<u,\phi>_\rho=\lambda \int_\Omega f|u|^{p-2}u\phi dx+\int_\Omega |u|^{2^*-2}u\phi dx\] for all $\phi\in \mathcal{X}^{1,2}(\Omega)$.
	\end{definition}
	
	To study the problem \eqref{p} via variatonal methods, we define the associated energy functional $J_\lambda:\mathcal{X}^{1,2}(\Omega)\to \mathbb{R}$ by
	\begin{align*}
	J_\lambda(u)&=\frac{1}{2}\hat{M}(\rho(u)^2)-\lambda\frac{1}{p}\int_\Omega f |u|^p dx-\frac{1}{2^*}\int_\Omega |u|^{2^*}dx\\
	&=\frac{a}{2} \rho(u)^2+\frac{b}{2\theta}\rho(u)^{2\theta}-\lambda\frac{1}{p}\int_\Omega f |u|^p dx-\frac{1}{2^*}\int_\Omega |u|^{2^*}dx
	\end{align*}
where $\hat{M}(t)=\int_{0}^{t}M(s)ds=a t+\frac{b}{\theta}t^{\theta}$.  One can see that $J_\lambda$ is $C^1$ and in light of definition \ref{def} the critical points of $J_\lambda$ corresponds to the solutions of \eqref{p}.
 {Let us recall, for any $\Omega\subseteq\mathbb{R}^N$ arbitrary open set  the sharp constant for  the embedding $H^1_0(\Omega)$ into $L^{2^*}(\Omega)$ is given by
	\[S_{N}=\inf_{u\in H^1_0(\Omega)\setminus\{0\}}\frac{\int_{ \Omega}|\nabla u|^2 dx}{\left(\int_\Omega |u|^{2^*} dx\right)^\frac{2}{2^*}}.\]Also, the sharp constant   for a mixed type Sobolev inequality is give by
		\begin{equation}\label{bestSc}
			S_{N,s}(\Omega)=\inf \frac{\rho(u)^2}{\left(\int_\Omega |u|^{2^*} dx\right)^\frac{2}{2^*}}.
		\end{equation}In \cite{biagi1}, authors have proved the best constant  in the natural mixed Sobolev inequality is never achieved and it coincides with the one coming from the purely local one, i.e., $S_{N}=S_{N,s}(\Omega)$.}

	\section{Nehari Manifold and Fibering Maps}\label{vf} To study the minimizers of the energy functional $J_\lambda$  it  must be bounded below on the  space $\mathcal{X}^{1,2}(\Omega)$, which is lacking in our setup.  The natural constrained to study the minimization problem is  the Nehari manifold.  The Nehari set associated to the problem \eqref{p} is given by
	\[\mathcal{N}_{\lambda}=\{u\in  \mathcal{X}^{1,2}(\Omega) \setminus\{0\}: D_u J_\lambda(u)u=0 \}.\] Observe that the set $\mathcal{N}_\lambda$ contains all the critical points of $J_\lambda$.  For a fixed $0\neq u\in \mathcal{X}^{1,2}(\Omega)$, we introduce the  fibering map $m_{\lambda,u}:(0,\infty)\to \mathbb{R}$ given as $m_{\lambda,u}(t)=J_{\lambda}(tu)$. It is easy to see that  $u\in \mathcal{N}_\lambda$ if and only if $m'_{\lambda,u}(1)=0$. More general, $tu\in \mathcal{N}_\lambda$ if and only if $m'_{\lambda,u}(t)=0$. It is evident that we can split the set $\mathcal{N}_\lambda$ in the following disjoint decompositions
	\begin{align*}
	\mathcal{N}^+_\lambda=\{u\in \mathcal{N}_\lambda:m''_{\lambda,u}(1)>0\}, \  
	\mathcal{N}^-_{\lambda}=\{u\in \mathcal{N}_\lambda:m''_{\lambda,u}(1)<0\}, \ 
	\mathcal{N}^0_\lambda=\{u\in \mathcal{N}_\lambda:m''_{\lambda,u}(1)=0\},
	\end{align*}
	which  correspond to $t=1$ as local  minimum, maximum and the inflection point of the fibering map.
In order to study minimization problem over the Nehari decompositions, we aim to show these sets are nonempty. Observe that for any $u\in \mathcal{X}^{1,2}(\Omega)\setminus\{0\}$, with $\int_{ \Omega} f|u|^q dx>0$, $t(u)u\in \mathcal{N}^0_{\lambda}$ if the pair $(t=t(u),\lambda=\lambda(u))$ is unique solution of the system 
\begin{align}\label{nrq}
&a t^2\rho(u)^2+b t^{2\theta} \rho(u)^{2\theta}-\lambda t^p\int_{ \Omega} f|u|^p dx-t^{2^*}\int_{ \Omega} |u|^{2^*}=0,\nonumber \\
&2a t^2\rho(u)^2+2\theta bt^{2\theta} \rho(u)^{2\theta}-\lambda p t^p\int_{ \Omega} f|u|^p dx-2^*t^{2^*}\int_{ \Omega} |u|^{2^*}=0.
\end{align} The uniqueness of $\lambda(u)$ follows from uniqueness of $t(u)$ as  {nontrivial} zero of the following scalar equation 
\[h(t)=a(2-p)t^2 \rho(u)^2+b(2\theta-p)t^{2\theta}\rho(u)^{2\theta}-(2^*-p)t^{2^*}\int_{ \Omega}|u|^{2^*}.\] To show the uniqueness of $t(u)$, we define
	\[m_u(t)=a(2-p)\rho(u)^2+b(2\theta-p)t^{2\theta-2}\rho(u)^{2\theta}-(2^*-p)t^{2^*-2}\int_{ \Omega}|u|^{2^*}dx,\] where $h_u(t)=t^2m_u(t)$. Since,
	\[m_u'(t)=b (2\theta-2)(2\theta-p)t^{2\theta-1}\rho(u)^{2\theta}-(2^*-p)(2^*-2)t^{2^*-1}\int_{ \Omega}|u|^{2^*}dx,\] $m_u(t)$ has a unique critical point
	\[t^*=\left(\frac{b(2\theta-2)(2\theta-p)\rho(u)^{2\theta}}{(2^*-p)(2^*-2)\int_{ \Omega}|u|^{2^*}}\right)^\frac{1}{2^*-2\theta}.\]  Moreover, as $m'_u(t)>0$ as $t\to 0^+$ and $m'_u(t)<0$ as $t\to +\infty$, and we can conclude that, $m_u(t)$ has a unique zero $t(u)>t^*>0$. Consequently, $h_u(t)$ has a  {unique nontrivial  zero.}

Solving the above system \eqref{nrq}, we have following implicit form of the nonlinear generalized Rayleigh quotient  $\lambda(u)$ 
		\[\lambda(u)=\frac{a(2^*-2)t(u)^{2-p}\rho(u)^2+b(2^*-2\theta) t(u)^{2\theta-p} \rho(u)^{2\theta}}{(2^*-p)\int_{ \Omega}f|u|^pdx}.\] From expression of $\lambda(u)$ it is clear that $\lambda(u)$ is $0-$ homogeneous and $t(u)$ is $(-1)-$homogeneous. Note that while dealing with degenerate local or nonlocal Kirchhoff problems which such class of nonlinearity, we have  explicit representations for $\lambda(u)$ (see, for instance \cite{fcaa}).
		
		  	 We define the extremal value  for Nehari manifold method (see, \cite{yav}), by
			\begin{equation}\label{ev}
			\lambda^*=\inf_{{\mathcal{X}^{1,2}(\Omega)\setminus\{0\}}}\left\{\lambda(u):\int_{ \Omega} f |u|^p dx>0\right\}.
		\end{equation}
	
	In order to study the minimisation problem, we show in the following Proposition that the decompositions of the Nehari manifold are nonempty.

	\begin{prop}\label{fiber}
		Given $u\in \mathcal{X}^{1,2}(\Omega)\setminus\{0\}$. There are two possibilities:
		\begin{enumerate}
			\item [(a)]if $\int_{ \Omega} f|u|^pdx>0$, then  there exists two critical points $t^+(u), t^-(u)$ of the fibering map $m_{\lambda,u}$ such that $t^+_{\lambda}(u)u\in \mathcal{N}^+_{\lambda}$ and $t^-(u)u\in \mathcal{N}^-_{\lambda}$. Moreover, $\phi_{\lambda,u}$ is decreasing in $(0,t^+(u)]$ and $[t^-(u),\infty)$ and increasing in $[t^+(u),t^-(u)]$ for any {$\lambda\in (0,\lambda(u))$};
			\item [(b)] if $\int_{ \Omega} f|u|^p\leq 0$
			then there exists unique $t^*$ such that $t^*u\in \mathcal{N}^+_{\lambda}$ {for any $\lambda>0$};
		\end{enumerate}
	\end{prop}
	\begin{proof} For  fixed $u\in \mathcal{X}^{1,2}(\Omega)\setminus\{0\}$ define a map $\Phi_u:(0,\infty)\to \mathbb{R}$ by
		\[\Phi_u(t)=at^{2-p}\rho(u)^2+bt^{2\theta-p}\rho(u)^{2\theta}-t^{2^*-p}\int_\Omega  |u|^{2^*}dx.\] Then it is clear that $tu\in \mathcal{N}_\lambda$ if and only if $t$ is root of the scalar equation
		\[\Phi_u(t)=\lambda\int_\Omega f |u|^pdx.\] We have $\Phi_u(t)>0^+$ as $t\to 0^+$ and $\Phi(t)\to -\infty$ as $t\to \infty$.  
		Using
		\[\Phi'_u(t)=a(2-p)t^{1-p}\rho(u)^2+b(2\theta-p)t^{2\theta-p-1}\rho(u)^{2\theta}-(2^*-p)t^{2^*-p-1}\int_{ \Omega}|u|^{2^*}dx,\] one can observe that $m''_{\lambda,tu}(1)=t^{p-1}\Phi'_u(t)$, where  	
		 $\Phi'_u(t)=t^{-1-p} h_u(t)$. 
			Thus, it is enough to analyze the nature of $h_u(t)$. We already know that $h_u(t)$ has unique zero at $t(u)$, therefore,
 $\Phi_u(t)$ has a global maximum at $t(u)$ in light of $\Phi'_u(t)>0$ when $t\to 0^+$ and $\Phi'_u(t)<0$ for large $t$. Consequently, when $\lambda<\lambda(u),$ using the relation $m''_{\lambda,tu}(1)=t^{p-1}\Phi'_u(t)$, there exists $t^\pm(u)$ satisfying $(a)$. In case when $\int_{ \Omega} f|u|^p dx\leq 0$ from the above analysis of the map $\Phi_u$, we can conclude the case $(b)$. This completes the proof. 
	\end{proof}
	{Following remark is a direct consequence of Proposition \ref{fiber}.  Moreover, as a consequence of implicit function theorem  the set $\mathcal{N}_\lambda$ is a  manifold.
		\begin{rem}\label{emptya}
			For all $\lambda>0$, $\mathcal{N}^\pm_{\lambda}\neq \emptyset$. Moreover, $\mathcal{N}^0_{\lambda}=\emptyset$ for all $\lambda<\lambda^*$.
		\end{rem}
		\subsection*{An estimate of extremal value $\lambda^*$:} Let $u\in \mathcal{X}^{1,2}(\Omega)\setminus\{0\}$ be fixed and for $t>0$, define a map
		\[\tilde{\Phi}_u(t)=b t^{2\theta-p}\rho(u)^{2\theta}-t^{2^*-p}\int_{\Omega}|u|^{2^*}dx.\] Then, we have $\max_{t>0} \tilde{\Phi}_u(t)= \tilde{\Phi}_u(\tilde{t}_{\max})$, where
		\[\tilde{t}_{\max}=\left(\frac{b(2\theta-p)\rho(u)^{2\theta}}{(2^*-p)\int_{\Omega}|u|^{2^*}dx}\right)^\frac{1}{2^*-2\theta}\geq  \left(\frac{b(2\theta-p)}{(2^*-p)} S_{N,s}(\Omega)^\frac{2^*}{2}\right)^\frac{1}{2^*-2\theta}\frac{1}{\rho(u)}:=\tilde{t}^0_{\max}.\]Therefore, as $\tilde{\Phi}_u$ is increasing from $[0,\tilde{t}_{\max}]$, we get
		\[\max_{t>0}\Phi_u(t)\geq \tilde{\Phi}_u(\tilde{t}^0_{\max})\geq \left(\frac{2^*-2\theta}{2\theta-p}\right)\left(\frac{b(2\theta-p)}{2^*-p}\right)^\frac{2^*-p}{2^*-2\theta} (S_{N,s}(\Omega)^\frac{2^*}{2})^\frac{2\theta-p}{2^*-2\theta}\rho(u)^{p}.\] Thus if,
		\[\lambda<\Lambda_1:=\left(\frac{2^*-2\theta}{2\theta-p}\right)\left(\frac{b(2\theta-p)}{2^*-p}\right)^\frac{2^*-p}{2^*-2\theta} (S_{N,s}(\Omega)^\frac{2^*}{2})^\frac{2\theta-p}{2^*-2\theta} \frac{S_{N,s}(\Omega)^\frac{p}{2}}{ \|f\|_{L^\frac{2^*}{2^*-2}(\Omega)}},\] it holds 
		\[\Phi_u(\tilde{t}^0_{\max})>\lambda\int_{ \Omega} f|u|^{p}, \ \ \forall  \ \lambda\in (0,\Lambda_1).\]  Consequently, for every $u\in \mathcal{X}^{1,2}(\Omega)\setminus\{0\}$, $t^\pm(u)u\in \mathcal{N}^\pm_{\lambda}$ and $\mathcal{N}^0_\lambda=\emptyset$. Thus, $\lambda^*\geq \Lambda_1$.
	}

		We have following observation on $J_\lambda$.
	
	\begin{lem}\label{coercive}
		The energy functional $J_\lambda$ is coercive and bounded below on $\mathcal{N}_\lambda$.
	\end{lem}
	\begin{proof}
		Take $0\neq u\in \mathcal{N}_\lambda$, we have 
		\begin{align*}
			J_\lambda(u)&=\frac{a}{2} \rho(u)^2+\frac{b}{2\theta}\rho(u)^{2\theta}-\frac{\lambda}{p}\int_\Omega f|u|^pdx-\frac{1}{2^*}\int_\Omega  |u|^{2^*}dx\\ 
			&\geq \left(\frac{1}{2\theta}-\frac{1}{2^*}\right)b\rho(u)^{2\theta}-\lambda\left(\frac{1}{p}-\frac{1}{2^*}\right)\int_\Omega f |u|^pdx\\
			&\geq \left(\frac{1}{2\theta}-\frac{1}{2^*}\right)b\rho(u)^{2\theta}-\lambda\left(\frac{1}{p}-\frac{1}{2^*}\right) S_{N,s}(\Omega)^{-p/2}\rho(u)^p \|f\|_{L^{\frac{2^*}{2^*-p}}(\Omega)}
		\end{align*} which implies $J_\lambda$ is coercive. By defining a map
		\[g(t)=\left(\frac{1}{2\theta}-\frac{1}{2^*}\right) t^{2\theta}-\lambda\left(\frac{1}{p}-\frac{1}{2^*}\right) S_{N,s}(\Omega)^{-p/2}t^p \|f\|_{L^{\frac{2^*}{2^*-p}}(\Omega)}\] one can observe that $g(t)$ is bounded below. In fact from here we can conclude that there exists $C>0$ such that $J_\lambda(u)>-C$. This completes the proof.
	\end{proof}

	The following lemma ensures that the local minimizers of the energy functional $J_\lambda$ on $\mathcal{N}_\lambda$ are critical points for $J_\lambda$, (see, Theorem 2.3  \cite{tf}).
	\begin{lem}\label{soln}
		If $u$ is a local minimizer of $J_\lambda$ in $\mathcal{N}_\lambda$ and $u\notin\mathcal{N}^0_\lambda$. Then $u$ is a critical point for $J_\lambda$.
	\end{lem}

	Now we are ready to introduce the minimization problem. Define 
	\[J_\lambda=\inf \{J_{\lambda}(u):u\in \mathcal{N}_\lambda\}, J_\lambda^+=\inf \{J_{\lambda}(u):u\in \mathcal{N}_\lambda^+\}, J_\lambda^-=\inf \{J_{\lambda}(u):u\in \mathcal{N}_\lambda^-\}.\] 
	
  In the upcoming Lemmas and Proposition, we have proved some technical results required to study the above minimization problems.
	\begin{lem}\label{close} {There exists $\delta>0$ such that $\rho(u)\geq \delta$ for $u\in \mathcal{N}_\lambda^-$. Moreover, $\mathcal{N}^-_\lambda$ is closed in topology of $\mathcal{X}^{1,2}(\Omega)$ for $\lambda\in (0,\lambda^*)$}.
	\end{lem}
	\begin{proof}
		If $u\in \mathcal{N}_\lambda^-$, we have
		\[(2\theta-p)b\rho(u)^{2\theta}<(2^*-p)\int_\Omega |u|^{2^*}dx\leq (2^*-p) S_{N,s}(\Omega)^{-2^*/2}\rho(u)^{2^*}.\] Which implies,
		$\rho(u)>\left(\frac{1}{ S_{N,s}(\Omega)^{-2^*/2}}\frac{2\theta-p}{2^*-p}\right)^\frac{1}{2^*-2\theta},$ thus,  $\rho(u)\geq \delta $ for some $\delta>0$. To prove the reaming part consider a sequence $\{u_k\}\subset\mathcal{N}^-_{\lambda}$ such that $u_k\to u$
		in $\mathcal{X}^{1,2}(\Omega)$. Then, $u\in \mathcal{N}^{-}_{\lambda}\cup \{0\}$ as $\mathcal{N}^0_{\lambda}=\emptyset$ for $\lambda\in (0,\lambda^*)$. As  $\rho(u)=\lim_{n\to \infty} \rho(u_k)\geq \delta$, we get $u\neq 0$ and $u\in \mathcal{N}^-_{\lambda}$. \end{proof}
	\begin{lem}\label{nbdp}
		For each $u\in \mathcal{N}_\lambda$ and $\lambda\in(0,\Lambda_1) $ there exists $\epsilon>0$ and a differentiable map $\xi: B(0,\epsilon)\subset\mathcal{X}^{1,2}(\Omega)\to \mathbb{R}$ such that $\xi(v)(u-v)\in \mathcal{N}_\lambda$ and .
		\begin{equation}\label{nbd}
			<\xi'(0),v>=\frac{2<u,v>_{\rho}-\lambda\int_\Omega f |u|^{p-2}uvdx-2^*\int_\Omega  |u|^{2^*-2}uvdx}{(2-p)\rho(u)^2-(2^*-p)\int_\Omega  |u|^{2^*}dx},
		\end{equation}
	\end{lem}
	\begin{proof}
		Fixed $u\in \mathcal{N}_\lambda$ and define a map $F_u:\mathbb{R}^+ \times  \mathcal{X}^{1,2}(\Omega) \to \mathbb{R}$ by
		\[F_u(t,v)=t^2\rho(u-v)^2-\lambda  t^p \int_\Omega f |u-v|^pdx-t^{2^*}\int_\Omega 
		|u-v|^{2^*}dx, \] then $F_u(1,0)=0$ and $\frac{\partial F_u}{\partial t}(1,0)\neq 0$ for $\lambda\in(0,\Lambda_1)$. Using implicit function theorem there exists $\epsilon>0$ and a differentiable functional $\xi:B(0,\xi)\subset \mathcal{X}^{1,2}(\Omega)\to \mathbb{R}$ such that $\xi(0)=1$, \eqref{nbd} holds
		and $F_u(\xi(v),v)=0$ for all $v\in B(0,\epsilon)$. Hence, $\xi(v)(u-v)\in \mathcal{N}_\lambda$. 
	\end{proof}
From Lemma \ref{coercive}, we know that $J_\lambda$ is bounded below in $\mathcal{N}_\lambda$. Therefore, for any $\lambda\in (0,\Lambda_1)$ by   Ekeland variational principle, there exists a minimizing sequence $\{u_k\}\subset\mathcal{N}_\lambda$ such that 
\[J_\lambda(u_k)\leq J_\lambda+\frac{1}{k}, \ \textrm{and} \ 
J_\lambda(u_k)\leq J_\lambda(v)+\frac{1}{k}\rho(u_k-v), \ \textrm{for} \ \textrm{all} \ v\in\mathcal{N}_\lambda.\] The following result is a consequence of  Ekeland variational principle and the Lemma \ref{nbdp}. The idea of the proof is similar to \cite{mana}, Proposition 3.8 and for this reason it has been omitted here.
	\begin{lem}\label{pscseq}
		For $\lambda\in(0,\Lambda_1)$ there exists a minimizing sequence $\{u_k\}\subset\mathcal{N}_\lambda$ such that
		\[J_\lambda(u_k)=J_\lambda+o_k(1), \ \ \textrm{and} \ \  J_\lambda^{'}(u_k)=o_k(1). \]
	\end{lem} 
	Next result insures that compactness of $J_\lambda$ can be recovered below a suitable value. 
	\begin{prop}\label{cpt}
		Let $\{u_k\}$ be a sequence in $\mathcal{X}^{1,2}(\Omega)$ such that
		\begin{equation}\label{psc}
		J_\lambda(u_k)\to c, J'_\lambda(u_k)\to 0, \textrm{as} \ n\to \infty,
		\end{equation} where $c<c_\lambda:=\frac{1}{N}(aS_{N,s}(\Omega))^\frac{N}{2}-\lambda^\frac{2\theta}{2\theta-p}\left(\frac{2\theta-p}{2^*p 2\theta}\right)\frac{\left((2^*-p)S_{N,s}(\Omega)^\frac{-p}{2}\|f\|_{L^\frac{2^*}{2^*-p}(\Omega)}\right)^\frac{2\theta}{2\theta-p}}{((2^*-2\theta)b)^\frac{p}{2\theta-p}}$,  then $\{u_k\}$ possesses a strongly convergent subsequence. 
	\end{prop}
	\begin{proof}
		Let $\{u_k\}$ be  a  sequence of bounded functions in $\mathcal{X}^{1,2}(\Omega)$ (as $J_\lambda$ is bounded below and coercive in $\mathcal{N}_\lambda$). { Then, by using compact embedding of $\mathcal{X}^{1,2}(\Omega)\hookrightarrow L^r(\Omega)$ for $r\in [1,2^*)$, there exists $u_0\in \mathcal{X}^{1,2}(\Omega)$) such that,  upto  a subsequence again denoting by $\{u_k\}$,  $u_k\rightharpoonup u_0 $  in $\mathcal{X}^{1,2}(\Omega), \rho(u_k)\to \mu$, $u_k\to u_0$ in $L^r(\Omega)$ for $r\in [1,2^*)$}, $u_k(x)\to u_0(x)$ a.e. in $\Omega$. If $\mu=0$, then it follows that $u_k\to 0$ in $\mathcal{X}^{1,2}(\Omega)$. Thus assume $\mu>0$. By the Brezis-Lieb lemma \cite{lieb}, we have 
		\begin{align}\label{blib}
		&\rho(u_k)^2=\rho(u_k-u_0)^2+\rho(u_0)^2+o_k(1),\nonumber\\
		\int_\Omega &|u_k|^{2^*}dx=\int_\Omega |u_k-u_0|^{2^*}dx+\int_\Omega |u_0|^{2^*}dx+o_n(1).
		\end{align}
		Testing \eqref{def} with $(u_k-u_0)$ and using \eqref{blib}, we have 
		\begin{align*}
			o_k(1)&=(a+b(\rho(u_k)^{2\theta-2}))\bigg(\int_{\mathbb{R}^{n}}\nabla u_k \nabla (u_k-u_0)dx\\
				&\quad+\int\int_{\mathbb{R}^{2N}}\frac{(u_k(x)-u_k(y))((u_k-u_0)(x)-(u_k-u_0)(y))}{|x-y|^{N+2s}}dxdy\bigg)\\
			&\quad-{\lambda\int_\Omega f |u_k|^{p-2}u_k(u_k-u_0)dx-\int_\Omega  |u_k|^{2^*-2}u_k(u_k-u_0)dx}\\
			&=(a+b\mu^{2\theta-2}){(\mu^2-\rho(u_0)^2)}-\lambda\int_\Omega f |u_k|^{p-2}u_k(u_k-u_0)-\int_\Omega  |u_k|^{2^*}dx+\int_\Omega  u_0^{2^*}dx\\
			&=(a+b\mu^{2\theta-2})\lim_{k\to \infty}\rho(u_k-u_0)^2-\lambda\int_\Omega f |u_k|^{p-2}u_k(u_k-u_0)dx-\int_\Omega  |u_k-u_0|^{2^*}dx. 
		\end{align*}Therefore, 
		\begin{equation}\label{e1}
			(a+b\mu^{2\theta-2})\lim_{k\to \infty}\rho(u_k-u_0)^2=\lambda\lim_{k\to \infty}\int_\Omega f |u_k|^{p-2}u_k(u_k-u_0)dx+\lim_{k\to \infty}\int_\Omega  |u_k-u_0|^{2^*}dx.
		\end{equation}
	Applying Lebesgue dominated convergence theorem, we have 
		\begin{equation}\label{eqn1}
			{(a+b\mu^{2\theta-2})\lim_{k\to \infty}\rho(u_k-u_0)^2= l^{2^*},}
		\end{equation} where, we denote  $\lim_{k\to \infty}\int_\Omega  |u_k-u_0|^{2^*}dx=l^{2^*}$. From the \eqref{e1} we can conclude that $l\geq 0$. If $l=0$, we have $u_k\to u_0$ and we are done in this case. Suppose the case when $l>0$. From, \eqref{bestSc}, we get
		\begin{equation}\label{enq2}
			\rho(u_k-u_0)^2\geq S_{N,s}(\Omega) l^{2}.
		\end{equation} {Using \eqref{eqn1} and \eqref{enq2}, we get}
		\begin{equation}\label{e2}
			l^{2^*-2}\geq S_{N,s}(\Omega)(a+b\mu^{2\theta-2}).
		\end{equation} Also, note that from \eqref{eqn1}, we have
		\begin{equation}\label{e3}
		(a+b\mu^{2\theta-2})(\mu^2-\rho(u_0)^2)\leq l^{2^*}.
		\end{equation} Using \eqref{e2} and \eqref{e3}, we obtain that
		\begin{equation}\label{e4}
			\mu^2\geq S_{N,s}(\Omega)^\frac{N}{2}a^\frac{2}{2^*-2}.
		\end{equation}
		For any $\phi\in \mathcal{X}^{1,2}(\Omega)$, denoting 
		\begin{align*}
		H(u_k,\phi)&=(a+b\rho(u_k)^{2\theta-2})\left(\int_{\mathbb{R}^N}\nabla u_k\nabla\phi dx+\int\int_{\mathbb{R}^{2N}}\frac{(u(x)-u(y))(\phi(x)-\phi(y))}{|x-y|^{N+2s}}dxdy\right)\\
		&\quad-\lambda\int_\Omega f |u_k|^{p-2}u_k\phi dx-\int_\Omega |u_k|^{2^*-2}u_k\phi dx.
		\end{align*}
		 Using Hölder inequality, we have
		\begin{align}\label{fs}
			J_\lambda(u_k)-\frac{1}{2^*}H(u_k,u_k)&=\left(\frac{1}{2}-\frac{1}{2^*}\right)a\rho(u_k)^2+\left(\frac{1}{2\theta}-\frac{1}{2^*}\right)b\rho(u_k)^{2\theta}-\lambda\left(\frac{1}{p}-\frac{1}{2^*}\right)\int_\Omega f|u_k|^{p}dx\nonumber\\
			&\geq \left(\frac{1}{2}-\frac{1}{2^*}\right)a\rho(u_k)^2+\left(\frac{1}{2\theta}-\frac{1}{2^*}\right)b\rho(u_k)^{2\theta}\\
			&\quad -\lambda\left(\frac{1}{p}-\frac{1}{2^*}\right) S_{N,s}(\Omega)^\frac{-p}{2}\|f\|_{L^\frac{2^*}{2^*-p}(\Omega)}\rho(u_k)^p.
		\end{align} Now define
		\[F_b(t)=\left(\frac{1}{2\theta}-\frac{1}{2^*}\right)b t^{2\theta}-\lambda\left(\frac{1}{p}-\frac{1}{2^*}\right)S_{N,s}(\Omega)^{\frac{-p}{2}}\|f\|_{L^\frac{2^*}{2^*-p}(\Omega)}t^p,\] then a simple computation implies that
		\begin{equation*}
		F_b(t)\geq -\lambda^\frac{2\theta}{2\theta-p}\left(\frac{2\theta-p}{2^*p 2\theta}\right)\frac{\left((2^*-p)S_{N,s}(\Omega)^\frac{-p}{2}\|f\|_{L^\frac{2^*}{2^*-p}(\Omega)}\right)^\frac{2\theta}{2\theta-p}}{((2^*-2\theta)b)^\frac{p}{2\theta-p}}.
		\end{equation*}
	Taking limit $k\to \infty$ \eqref{fs}, together with  \eqref{psc},\eqref{e4}, we get
		\[c\geq \frac{1}{N}(aS_{N,s}(\Omega))^\frac{N}{2}-\lambda^\frac{2\theta}{2\theta-p}\left(\frac{2\theta-p}{2^*p 2\theta}\right)\frac{\left((2^*-p)S_{N,s}(\Omega)^\frac{-p}{2}\|f\|_{L^\frac{2^*}{2^*-p}(\Omega)}\right)^\frac{2\theta}{2\theta-p}}{((2^*-2\theta)b)^\frac{p}{2\theta-p}},  \] which gives  a contradiction to the assumption. This completes the proof.
	\end{proof}

	\section{Existence of first solution in $\mathcal{N}^+_\lambda$}\label{sec4}
In this section, using the standard minimization argument, we prove  the existence of  first solution of the problem \eqref{p} in $\mathcal{N}^+_{\lambda}$. We start by proving  that the energy level is negative  in $\mathcal{N}^+_{\lambda}$.
	\begin{lem}\label{neg} 
	For $\lambda>0$, we have that $J_\lambda^+=\inf \{J_{\lambda}(u):u\in \mathcal{N}_\lambda^+\}<0$.
\end{lem}
\begin{proof}
	Let $0\neq u\in\mathcal{N}_\lambda\subset\mathcal{N}^+_{\lambda}$, we have 
	\begin{align*}
		J_\lambda(u)&=\left(\frac{1}{2}-\frac{1}{2^*}\right)a\rho(u)^2+\left(\frac{1}{2\theta}-\frac{1}{2^*}\right)b \rho(u)^{2\theta}-\lambda\left(\frac{1}{p}-\frac{1}{2^*}\right)\int_\Omega f |u|^pdx\\
		&<\left(\frac{1}{2}-\frac{1}{2^*}\right)a\rho(u)^2+\left(\frac{1}{2\theta}-\frac{1}{2^*}\right)b \rho(u)^{2\theta}\\
		&\quad-\left(\frac{1}{p}-\frac{1}{2^*}\right)\left(\frac{2^*-2}{2^*-p}a\rho(u)^2+\frac{2^*-2\theta}{2^*-p} b\rho(u)^{2\theta}\right)\\
		&=\frac{2^*-2}{2^*}\left(\frac{1}{2}-\frac{1}{p}\right)\rho(u)^2+\frac{2^*-2\theta}{2^*}\left(\frac{1}{2\theta}-\frac{1}{p}\right)b\rho(u)^{2\theta}<0,
	\end{align*} and as $J(0)=0$, we get $m^+_{\lambda}<0$.
\end{proof} Choose $\Lambda_2>0$ such that $\Lambda_2^\frac{2\theta}{2\theta-p}\left(\frac{2\theta-p}{2^*p 2\theta}\right)\frac{\left((2^*-p)S_{N,s}(\Omega)^\frac{-p}{2}\|f\|_{L^\frac{2^*}{2^*-p}(\Omega)}\right)^\frac{2\theta}{2\theta-p}}{((2^*-2\theta)b)^\frac{p}{2\theta-p}}\leq \frac{1}{N}(aS_{N,s}(\Omega))^\frac{N}{2}$.

	\textit{Proof of Theorem \ref{t1}:}
		Considering $\Lambda_0=\min\{\Lambda_1,\Lambda_2\}$ and using Proposition 
		\ref{pscseq}, we get a minimizing Palais-Smale sequence $\{u_k\}\subset\mathcal{N}^+_{\lambda}$ such that $\{u_k\}$ is bounded in $\mathcal{X}^{1,2}(\Omega)$ and $J_\lambda(u_k)\to m_{\lambda}^+$. Also, from Lemma \ref{neg}, we have $J_\lambda^+<0$. In view of the Proposition \ref{cpt}, there exists $u_0\in \mathcal{X}^{1,2}(\Omega)$ such that $u_k\to u_0$ in $\mathcal{X}^{1,2}(\Omega)$. Thus, $u_0$ is a minimizer of $J_\lambda$ in $\mathcal{N}_\lambda$ for all $\lambda\in(0,\Lambda_1)$ since $J_\lambda(u_0)<0$, we have that $u_0\not\equiv0$. Next, we claim that $u_0\in \mathcal{N}^+_\lambda$. If not then $u_0\in \mathcal{N}^0_{\lambda}$ which is not possible as $\mathcal{N}^0_\lambda=\emptyset$ for all $\lambda\in (0,\Lambda_0)$ (see, remark \ref{emptya}). Therefore, $u_0\in \mathcal{N}^+_{\lambda}$. Using Lemma  \ref{soln}, we can conclude that the obtained minimizer $u_0\in \mathcal{N}_\lambda^+$ is a critical point and equivalently the solution of the problem \eqref{p}. 
		
		Next, it remain to show that obtained solution is non-negative. It is required to be verified due  to presence of nonlocal fractional operator, we have $\rho(u)\neq \rho(|u|)$ in $\mathcal{X}^{1,2}(\Omega)$ and thus $J_\lambda(u)\neq J_\lambda(|u|)$. To overcome this difficulty, we consider positive part of the  problem and the  corresponding energy functional as follows 
		\[J^+_{\lambda}(u)=\frac{1}{2}\hat{M}(\rho(u)^2)-\lambda\frac{1}{p}\int_\Omega f (u^+)^p dx-\frac{1}{2^*}\int_\Omega (u^+)^{2^*}dx,\] where $u^+:=\max\{u,0\}$ is the positive part of $u$.  Then, it is easy to see that the critical points of $J_\lambda$ are also critical points of $J_\lambda^+$. Thus, 
		\begin{align}\label{positive}
		M(\rho( u_0)^2)&\left(\int_{\mathbb{R}^N}\nabla u_0\cdot \nabla \phi dx+\int\int_{\mathbb{R}^{2N}}\frac{(u_0(x)-u_0(y))(\phi(x)-\phi(y))}{|x-y|^{N+2s}}dxdy\right)\nonumber\\
		&=\lambda \int_\Omega f{(u_0^+)}^{p-1}\phi dx+\int_\Omega (u_0^+)^{2^*-1}\phi dx\end{align} for all $\phi\in \mathcal{X}^{1,2}(\Omega)$. Testing \eqref{positive} by $\phi=u^{-}$ and using the inequality
		\begin{align*}
		(u_0(x)-u_0(y))(u_0^-(x)-u_0^-(y))&=-u_0^+(x)u_0^-(y)-u_0^-(x)u_0^+(y)-(u_0^-(x)-u_0^-(y))^2\\
		&\leq -|u_0^-(x)-u_0^-(y)|^2,
		\end{align*}
		we have $\rho(u_0^-)=0$, thus $u_0$ is nonnegative solution of the problem \eqref{p}.  This completes the proof.

	\section{Existence of Second solution in $\mathcal{N}^-_\lambda$}\label{sec5}
	
	To show the existence of the second solution to the problem \eqref{p}, we have followed the idea and the estimates from \cite{tfwu,biagi1}. Consider the test function $\eta\in C^\infty_0(\Omega)$ such that 
	\[0\leq \eta(x)\leq 1, \ \textrm{in} \  \Omega, \ \eta_\frac{\rho}{2}(x)=1, \ \textrm{in }\ B_\frac{\rho}{2}(0), \ \textrm{and} \ \eta_\frac{\rho}{2}(x)=0 \ \textrm{in }\ (B_\frac{\rho}{2}(0))^c, \] for $\rho$ sufficiently small. For $\epsilon>0$, let
	\[u_{\epsilon}(x)=\frac{\epsilon^\frac{(N-2)}{2}}{\left(|x|^2+\epsilon^2\right)^\frac{N-2}{2}}, \ \textrm{and} \ u_{\epsilon,\eta}=\frac{\eta(x)u_\epsilon(x)}{\|\eta u_\epsilon\|_{L^{2^*}(\Omega)}}.\] Consider the following sets
	\[U_1=\left\{u\in \mathcal{X}^{1,2}(\Omega)\setminus\{0\}:\frac{1}{\rho(u)}t^{-}\left(\frac{u}{\rho(u)}\right)>1\right\}\cup\{0\},\] and 
		\[U_2=\left\{u\in \mathcal{X}^{1,2}(\Omega)\setminus\{0\}:\frac{1}{\rho(u)}t^{-}\left(\frac{u}{\rho(u)}\right)<1\right\},\] where $t^-$ is as in Proposition \ref{fiber}. Then $\mathcal{N}_\lambda^-=\left\{u\in \mathcal{X}^{1,2}(\Omega)\setminus\{0\}:\frac{1}{\rho(u)}t^{-}\left(\frac{u}{\rho(u)}\right)=1\right\}$ is a connected component of $U_1, U_2$. Also, one can observe that $\mathcal{N}_\lambda^+\subset U_1$ and thus, $u_0\in U_1$. Now, we provide the subsequent technical lemma. The idea of the proof is similar to \cite{tfwu}.
	\begin{lem}
		For all $\lambda\in (0,\Lambda_0)$ and $u_0$ be a local minimizer for the functional $J_\lambda$ in $\mathcal{X}^{1,2}$ obtained in Theorem \ref{t1}. Then, for any $\epsilon>0$ and $\eta$ defined above there exists $l_0>0$ such that $u_0+l_0 u_{\epsilon,\eta}\in U_2$.
	\end{lem}
	\begin{prop}\label{neglevl}
		Let $\lambda\in(0,\Lambda_0)$ and let $u_0$ be the local minimizer achieved in last case for the functional $J_\lambda$ in $\mathcal{X}^{1,2}(\Omega)$. Then, for any $r>0$ and $\eta$ there exists $\epsilon_0=\epsilon_0(r,N)$ and $\Lambda_3>0$ such that $J_\lambda(u_0+r u_{\epsilon,\eta})<c_\lambda$ for any $\epsilon\in (0,\epsilon_0)$  and $\lambda \in (0,\min\{\lambda_0,\Lambda_3\})$.
	\end{prop}
	
	\begin{proof} We have
		\begin{align}
			J_\lambda(u_0+r&u_{\epsilon,\eta})=\frac{a}{2}\left(\int_{\mathbb{R}^N}|\nabla u_0|^2+\int\int_{\mathbb{R}^{2N}}\frac{|u_0(x)-u_0(y)|^2}{|x-y|^{N+2s}}dxdy\right)+\frac{b}{2\theta}r\bigg(\int_{\mathbb{R}^N}\nabla u_0\nabla u_{\epsilon,\eta}dx\nonumber\\
			&+\int\int_{\mathbb{R}^{2N}}\frac{(u_0(x)-u_0(y))(u_{\epsilon,\eta}(x)-u_{\epsilon,\eta}(y))}{|x-y|^{N+2s}}dxdy\bigg)\nonumber\\
			&+r^2\left(\int_\Omega|\nabla u_{\epsilon,\eta}|^2dx+\int\int_{\mathbb{R}^{2N}}\frac{|u_{\epsilon,\eta}(x)-u_{\epsilon,\eta}(y)|}{|x-y|^{N+2s}}dxdy\right)\nonumber\\
			&=\frac{a}{2}\rho(u_0)^2+\frac{a}{2}r^2\rho(u_{\epsilon,\eta})^2+ar\bigg(\int_{\mathbb{R}^N}\nabla u_0\nabla u_{\epsilon,\eta}dx\\
			&\quad+\int\int_{\mathbb{R}^{2N}}\frac{(u_0(x)-u_0(y))(u_{\epsilon,\eta}(x)-u_{\epsilon,\eta}(y))}{|x-y|^{N+2s}}dxdy\bigg)\nonumber\\
			&+\frac{b}{2\theta}\left(\rho(u_0)^2+r^2\rho(u_{\epsilon,\eta})^2+2r\int\int_{\mathbb{R}^{2N}}\frac{(u_0(x)-u_0(y))(u_{\epsilon,\eta}(x)-u_{\epsilon,\eta}(y))}{|x-y|^{N+2s}}dxdy\right)^\theta\nonumber\\
			&-\frac{\lambda}{p}\int_\Omega f|u_0+ru_{\epsilon,\eta}|^pdx-\frac{1}{2^*}\int_\Omega |u_0+ru_{\epsilon,\eta}|^{2^*}dx. \nonumber
		\end{align}  In order to estimate terms in the above expression, we will using following inequalities
		\begin{equation}\label{inq1}
		(\alpha+\beta)^{\theta}\leq 2^{\theta-1}(\alpha^\theta+\beta^\theta)\leq \alpha^{\theta} +C_\theta(\alpha^\theta+\beta^\theta)+\theta\alpha^{\theta-1}\beta,   \ \ \textrm{for} \ \textrm{any} \ \alpha,\beta\geq 0, \theta\geq 1, 
		\end{equation}
 \begin{equation}\label{inq3}
	(\alpha+\beta)^{p}-\alpha^p-\beta^p-p\alpha^{p-1}\beta\geq C_1\alpha\beta^{p-1},  \ \ \textrm{for} \ \textrm{any} \ \alpha,\beta\geq 0,  p>2, 
	\end{equation} where $C_\theta>0$ depending on $\theta$ and $C_1>0$.
		 Using Young's inequity, for some $D_\theta>0$, we have
		\begin{align}\label{inq2}
			\frac{b}{2\theta}&\left(\rho(u_0+ru_{\epsilon,\eta})^{2\theta}\right)\leq \frac{b}{2\theta}\rho(u_0)^{2\theta}+b C_\theta(\rho(u_0))^{2\theta}+b D_\theta r^{2\theta}(\rho(u_{\epsilon,\eta}))^{2\theta}\nonumber\\
			&+b(\rho(u_0))^{2\theta-2} r \left(\int_{\mathbb{R}^N}\nabla u_0\nabla u_{\epsilon,\eta}dx+\int\int_{\mathbb{R}^{2N}}\frac{(u_0(x)-u_0(y))(u_{\epsilon,\eta}(x)-u_{\epsilon,\eta}(y))}{|x-y|^{N+2s}}dxdy\right).
		\end{align} 
		From \eqref{inq2}, we have 
		\begin{align*}
			J_\lambda(u_0+ru_{\epsilon,\eta})&\leq J_\lambda(u_0)+\frac{a}{2}r^2(\rho(u_{\epsilon,\eta}))^2+b C_\theta(\rho(u_0))^{2\theta}+b D_\theta r^{2\theta}(\rho(u_{\epsilon,\eta}))^{2\theta}\\
			&\quad- \frac{\lambda}{p} \int_\Omega f \left((u_0+ru_{\epsilon,\eta}^p)dx-u_0^p-pru_0^{p-1}u_{\epsilon,\eta}\right)\\
			&\quad-\frac{1}{2^*}\int_\Omega  \left((u_0+ru_{\epsilon,\eta}^{2^*})dx-u_0^{2^*}-{2^*}ru_0^{2^*-1}u_{\epsilon,\eta}\right).
		\end{align*}Denote, $\rho(u_0)=R$.  Using  $f>0$ in the support of $u_{\epsilon,\eta}$ and the inequality \eqref{inq1},\eqref{inq3}, we can conclude that
		\begin{align}
		J_\lambda(u_0+ru_{\epsilon,\eta})&\leq J_\lambda(u_0)+\frac{a}{2}r^2(\rho(u_{\epsilon,\eta}))^2+b C_\theta R^{2\theta}+b D_\theta r^{2\theta}(\rho(u_{\epsilon,\eta}))^{2\theta}\nonumber\\
		&-\frac{1}{2^*}r^{2^*}\int_\Omega  u_{\epsilon,\eta}^{2^*}dx-C_1 r^{2^*-1}\int_\Omega  u_{\epsilon,\eta}^{2^*-1}dx.
		\end{align}Considering the estimates from \cite{biagi1} and taking $b=\epsilon^q$ {with $q>N-2$}, we get
		\begin{equation}
			J_{\lambda}(u_0+ru_{\epsilon,\eta})\leq \frac{a}{2} r^2(S_{N,s}(\Omega)+O(\epsilon^{k_{s,N}}))+C_2 \epsilon^q+C_3 \epsilon^{N-2}-\frac{r^{2^*}}{2^*}-C_4 r^{2^*-1} \epsilon^{(N-2)/2}
		\end{equation}  where $k_{s,N}=\min\{N-2,2-2s\}$. Next,   define
		\[G(t)=\frac{a}{2} t^2(S_{N,s}(\Omega)+O(\epsilon^{k_{s,N}}))-\frac{t^{2^*}}{2^*}-C_4 t^{2^*-1} \epsilon^{(N-2)/2}.\]
			{Observe that $G(t)\to -\infty$ as $t\to \infty$ and $G(t)\to 0^+$ as $t\to 0^+$, thus there exists $t_\epsilon$ such that $\frac{d}{dt}G(t)|_{t=t_\epsilon}=0$ and $G(t_\epsilon)=\sup_{t\geq 0} G(t)$. Also $G'(t_\epsilon)=0$ implies there exists $\nu>0$ such that $t_\epsilon\geq \nu>0$.} Using the fact $\inf J_\lambda=m^+_{\lambda}<0$ in $\mathcal{N}^+_{\lambda}$, we get
	\begin{align*}
		J_\lambda(u_0+r u_{\epsilon})&\leq \frac{a}{2} t^2(S_{N,s}(\Omega)+O(\epsilon^{k_{s,N}}))-\frac{t^{2^*}}{2^*}-C_4 t^{2^*-1} \epsilon^{(N-2)/2}+C_5 \epsilon^{N-2}\nonumber\\
		&\leq  \frac{a}{2} t^2(S_{N,s}(\Omega)+O(\epsilon^{k_{s,N}}))-\frac{t^{2^*}}{2^*}+C_5 \epsilon^{N-2}-C_6  \epsilon^{(N-2)/2}\noindent,
	\end{align*}  where $C_5,C_6>0$ are positive constants independent of $\epsilon,\lambda$. Since, the map $t\to \frac{a}{2} t^2(S_{N,s}(\Omega)+O(\epsilon^{k_{s,N}}))-\frac{t^{2^*}}{2^*}$ is increasing in $[0, (a(S_{N,s}(\Omega)+C\epsilon^{k_{s,N}}))^\frac{1}{2^*-2})$, we have 
	\begin{align*}
	J_\lambda(u_0+ru_\epsilon)&\leq \left(\frac{1}{2}-\frac{1}{2^*}\right)(a(S_{N,s}(\Omega)+O(\epsilon^{k_{s,N}})))^{\frac{2^*}{2^*-2}}+C_5 \epsilon^{N-2}-C_6 \epsilon^{(N-2)/2}\\
	&\leq \frac{1}{N}(a S_{N,s}(\Omega))^\frac{N}{2}+C_7\epsilon^{\min\{N-2,2-2s\}}-C_5\epsilon^\frac{N-2}{2},
	\end{align*} where $C_7>0$ and $\epsilon$ is sufficiently small. Choosing  $0<\epsilon<\epsilon_1$ sufficiently small, provided
	\[C_7\epsilon^{\min\{N-2,2-2s\}}-C_5\epsilon^\frac{N-2}{2}<0.\] Therefore,
 taking $\lambda\in (0,\Lambda_3)$, where $\Lambda_3$ satisfies the following inequality
 \[{\Lambda_3}^\frac{2\theta}{2\theta-p}\left(\frac{2\theta-p}{2^*p 2\theta}\right)\frac{\left((2^*-p)S_{N,s}(\Omega)^\frac{-p}{2}\|f\|_{L^\frac{2^*}{2^*-2}(\Omega)}\right)^\frac{2\theta}{2\theta-p}}{((2^*-2\theta)b)^\frac{p}{2\theta-p}}<C_7\epsilon^{\min\{N-2,2-2s\}}-C_5\epsilon^\frac{N-2}{2},\] we have that $J_\lambda(u_0+ru_{\epsilon})<c_\lambda$, where $c_\lambda$ is defined in proposition  \ref{cpt}. This completes the proof. 
\end{proof}
\textit{Proof of Theorem \ref{t2}:} Assume $\Lambda_0$ and $\Lambda_3$ as it is in the Section \ref{sec4} and  Proposition \ref{neglevl} and   fix $\lambda<\Lambda_{00}:=\min\{\Lambda_0,\Lambda_3\}$.  Also as we have $u_0\in U_1$ and $u_0+l_0u_{\epsilon,\eta}\in U_2$,  we can define a continuous path connecting $U_1$ and $U_2$ as $ t\to\gamma(t):=u_0+tl_0u_{\epsilon,\eta}$. Therefore, there exists $t\in (0,1)$ such that $\gamma(t)\in \mathcal{N}^-_{\lambda}$ and consequently $m^-_{\lambda}=\inf_{u\in \mathcal{N}_\lambda^-}J_\lambda(u)\leq J_\lambda(\gamma(t))$. Furthermore, from Proposition \ref{neglevl} we get $J_\lambda^-<c_\lambda$ for $\lambda<\Lambda_{00}$. Again from Lemma \ref{pscseq} can ensure existence of a minimizing Palais–Smale sequence $\{u_k\}\subset\mathcal{N}^-_\lambda$ such that $J_\lambda(u_k)\to m^-_\lambda$. Since $J_\lambda^+<c_\lambda$, thus, from Proposition \ref{cpt} there exists $u_1\in \mathcal{X}^{1,2}(\Omega)$ such that $u_k\to u_1$ in $\mathcal{X}^{1,2}(\Omega)$. Therefore, $u_1$ is a minimizer of $J_\lambda$.  As from Lemma \ref{close} $\mathcal{N}^-_{\lambda}$ is closed we have that $u_1\in \mathcal{N}^-_{\lambda}$ for all $\lambda\in (0,\Lambda_{00})$. Since $\mathcal{N}_\lambda^+\cap\mathcal{N}^-_\lambda=\emptyset$, we can conclude that $u_0,u_1$ are distinct. This completes the proof of the Theorem \ref{t2}. Again following the arguments as in the proof of Theorem \ref{t1}, we can conclude that $u_1$ is a non-negative solution of the problem \eqref{p}. This ends the proof.

\begin{rem} 	{We believe that in case the sublinear term in the problem \eqref{p} get perturbed with  subcritical growth  then {using compact compact embedding} $\mathcal{X}^{1,2}(\Omega)\hookrightarrow L^r(\Omega)$ for every $r\in[1,2^*)$ we can establish the multiplicity result in $(0,\Lambda^*)$, where
		\[ \Lambda^*=\inf_{{\mathcal{X}^{1,2}(\Omega)\cap \int_{ \Omega} f |u|^p dx>0}}\frac{a(q-2)t(u)^{2-p}\rho(u)^2+b(q-2\theta) t(u)^{2\theta-p} \rho(u)^{2\theta}}{(q-p)\int_{ \Omega}f|u|^pdx},\] for $1<p<2<q<2^*, \theta\in [1,\frac{q}{2})$ and $ q\in (2\theta,2^*)$.
		Precisely, in this case  we do not required, the Proposition \ref{cpt}. Moreover, following the idea in \cite{fcaa,silva}, one can ensure the multiplicity when $\lambda\in (0,\Lambda^*+\epsilon)$, for some $\epsilon>0$. Note that,  when $\lambda\geq \Lambda^*$, the Nehari set $\mathcal{N}_\lambda$ is no longer a manifold.
		}
\end{rem}

\noindent \textbf{Acknowledgment}

The author would like to thank Dr. Pawan Kumar Mishra for  providing incisive feedback on the results and discussions. 

\noindent\textbf{Author declaration of interest}

The author declares that he has no known
competing financial interests or personal relationships that could have appeared to
influence the work reported in this paper.

\noindent\textbf{Data Availability Statement}

Data sharing not applicable to this article as no
datasets were generated or analysed during the current study.

\end{document}